\documentclass{amsart}
\usepackage[english]{babel}
\usepackage{amssymb}
\usepackage{hyperref}   
\usepackage{mathtools}
\mathtoolsset{showonlyrefs}

\newtheorem{theorem}{Theorem}

\begin{document}

\title{Mass and polyhedra in asymptotically hyperbolic manifolds}

\author{Xiaoxiang Chai}
\address{Korea Institute for Advanced Study, Seoul 02455, South Korea}
\email{xxchai@kias.re.kr}

\begin{abstract}
  Using the upper half space model, we evaluate a component of the hyperbolic
  mass functional evaluated on a special family of polyhedra extending a
  formula of Miao-Piubello.
\end{abstract}

{\maketitle}

\section{Introduction}

We say that a three dimensional manifold $(M, g)$ is an asymptotically
hyperbolic manifold if outside a compact set $M$ is diffeomorphic to the
standard hyperbolic space $(\mathbb{H}^3, \bar{g})$ minus a geodesic ball and
\begin{equation}
  |e|_{\bar{g}} + | \bar{\nabla} e|_{\bar{g}} + | \bar{\nabla} \bar{\nabla}
  e|_{\bar{g}} = O (\mathrm{e}^{- \tau r}),
\end{equation}
where $\bar{\nabla}$ is the connection on $\mathbb{H}^3$, $r$ is the distance
function to a fixed point $o$ and $\tau > \tfrac{3}{2}$. We fix $o$ to be the
point $(0, 0, 1)$ in the upper half space model
\begin{equation}
  \bar{g} = \tfrac{1}{(x^3)^2} ((\mathrm{d} x^1)^2 + (\mathrm{d} x^2)^2 +
  (\mathrm{d} x^3)^2) .
\end{equation}
Then
\begin{equation}
  2 \cosh r = \tfrac{1}{x^3} ((x^1)^2 + (x^2)^2 + (x^3)^2 + 1) .
  \label{distance formula}
\end{equation}
See {\cite[Chapter A]{benedetti-lectures-1992}}. We assume in this article
that $M$ is diffeomorphic to $\mathbb{H}^3$. Using the upper half space
coordinates, we say that $\mathbf{{\Delta}}$ is a {\itshape{polyhedron}} if
each of its faces is a plane in the Euclidean sense whatever the metric the
upper half space carries.

Let $V = \tfrac{1}{x^3}$ and $F$ be a face of the polyhedron, let $\bar{\nu}$
be the $\bar{g}$-normal pointing outward of $\mathbf{{\Delta}} :=
\mathbf{{\Delta}}_q$, we see that
\begin{equation}
  \bar{\nu} = x^n a^i \partial_i
\end{equation}
where $a^i$ are constants and $a$ is a vector of length one under the
Euclidean metric. Due to conformality of $\bar{g}$ to the Euclidean metric
$\delta$, faces of $\mathbf{{\Delta}}$ meets at constant angles. Easily,
\begin{equation}
  \partial_{\bar{\nu}} V = x^3 a^i \partial_i (\tfrac{1}{x^n}) = - a^3
  \tfrac{1}{x^3} = - a^3 V. \label{normal derivative of static potential}
\end{equation}
Note that $a^3 \in [- 1, 1]$ and the case of $a^3 = \pm 1$ is used in the
works of Jang-Miao {\cite{jang-hyperbolic-2021}} and the author
{\cite{chai-asymptotically-2021}}. When $|a^3 | < 1$, this faces lies in a
so-called equidistant hypersurface. Suppose that the face lies in a plane
which intersect the $x^3$ axis at $x^3 = z_0$, then following
{\cite{chai-asymptotically-2021}}, the vector $X = x - z_0 \partial_n$ is
tangent to the face and $\ensuremath{\operatorname{div}} (x - z_0 \partial_n)
= z_0 V$. We calculate the second fundamental form of each face under the
metric $b$. Pick a local coordinate $y^{\alpha}$ where $\alpha = 1, 2$. Using
the conformality to $\delta$,
\begin{align}
  & \bar{A}_{\alpha \beta} \\
  = & \langle \partial_{\alpha}, \partial_{\beta} \rangle_{\delta} x^3
  \partial_a \tfrac{1}{(x^3)^2} \\
  = & - a^3 \tfrac{1}{(x^3)^2} \langle \partial_{\alpha}, \partial_{\beta}
  \rangle_{\delta} \\
  = & - a^3 \langle \partial_{\alpha}, \partial_{\beta} \rangle_b .
  \label{second fundamental form background} 
\end{align}
Each face $F$ is umbilic and the mean curvature is then $\bar{H} = - 2 a^3$.
Euclidean spheres are also umbilic in the hyperbolic metric, however, they do
not satisfy the condition {\eqref{normal derivative of static potential}}.
Obviously, these discussions works for higher dimension

The mass integrand for an asymptotically hyperbolic manifold (See
{\cite{chrusciel-mass-2003}}) is
\begin{equation}
  \mathbb{U}= V\ensuremath{\operatorname{div}}e - V \mathrm{d}
  (\ensuremath{\operatorname{tr}}_b e) +\ensuremath{\operatorname{tr}}_b e
  \mathrm{d} V - e (\bar{\nabla} V, \cdot) .
\end{equation}
We assume that for a family of polyhedra indexed by $q$ the mass satisfies the
following
\[ \mathbf{M} (V) = \int_{\partial \mathbf{{\Delta}}_q} \mathbb{U}^i
   \bar{\nu}_i \mathrm{d} \bar{\sigma} + o (1), \label{mass} \]
where $\bar{\nu}$ is the $b$-normal to the face of $\mathbf{{\Delta}}_q$ and
$\mathrm{d} \bar{\sigma}$ is the two dimensional volume element.

Such a family is easy to find. For example, according to
{\cite{chrusciel-mass-2003}} or {\cite{michel-geometric-2011}}, if each
polyhedron of the family is enclosed by a geodesic sphere and enclose another
geodesic sphere, radius of each sphere goes to infinity as
$\mathbf{{\Delta}}_q$ exhaust the manifold $M$, then such a family provides an
example. Let $E_q$ be the all edges of $\mathbf{{\Delta}}_q$, $\alpha$ be the
dihedral angle for by neighboring faces along $E_q$. We denote by $\mathrm{d}
v$, $\mathrm{d} \sigma$ and $\mathrm{d} \lambda$ respectively the three, two
and one dimensional volume element. We put a bar over a letter to indicate the
quantity is calculated with respect to the background metric $\bar{g}$.

\begin{theorem}
  \label{miao}Assuming each dihedral angle satisfies the bound $\sin
  \bar{\alpha} \geqslant c > 0$, then the mass $\mathbf{M} (V)$ is
  \begin{align}
    & \mathbf{M} (V) \\
    = & - \int_{\partial \mathbf{{\Delta}}_q} 2 V (H - \bar{H}) \mathrm{d}
    \bar{\sigma} + 2 \int_{E_q} V (\alpha - \bar{\alpha}) \mathrm{d}
    \bar{\lambda} \\
    & + \int_{\partial \mathbf{{\Delta}}_q} O (\cosh^{- 2 \tau + 1} r)
    \mathrm{d} \bar{\sigma} + \int_{E_q} O (\cosh^{- 2 \tau + 1} r) \mathrm{d}
    \bar{\lambda} + o (1) . \label{miao-piubello} 
\end{align}
\end{theorem}

As suggested by Miao {\cite{miao-measuring-2020}}, this type of formula may be
used to promote the Gromov dihedral rigidity {\cite{gromov-dirac-2018}} to an
integrated form. However, in the asymptotically flat case, one can perturb
graphically of a face and use Taylor expansion to the third order to see a
counterexample.

Also, the Miao-Piubello type formula {\eqref{miao-piubello}} leads us to do
comparison between a Riemannian polyhedron with the polyhedron whose faces are
realized as Euclidean planes in the upper half space model. In particular, one
could consider the Gromov type dihedral rigidity. The case when the model is a
cone type polyhedron with base faces lying on a horosphere is possible using
the method developed by {\cite{li-polyhedron-2020}}. Instead of minimal
surface with capillary angle condition, one uses constant mean curvature two
surfaces with capillary angle condition.

\section{Example and proof}

Before we show the proof of Theorem \ref{miao}, we calculate an easy example
other than the parabolic cylinder {\cite{jang-hyperbolic-2021}} to illustrate
that the terms $\int_{\partial \mathbf{{\Delta}}_q} \cosh^{- 2 \tau + 1} r
\mathrm{d} \bar{\sigma}$ and $\int_{E_q} \cosh^{- 2 \tau + 1} r \mathrm{d}
\bar{\lambda}$ can be $o (1)$. We assume the base face $B$ is a regular
$n$-side polygon lies at the horosphere $\{x^3 = \varepsilon\}$ and centered
at $(0, 0, \varepsilon)$ for $\varepsilon$ small, the apex is $(0, 0,
\varepsilon^{- 1})$. The distance between the vertex of the polygon to the
$x^3$-axis is $\rho (\varepsilon)$. We assume that
\begin{equation}
  \rho (\varepsilon) = o (\varepsilon^{- 2 \tau})
\end{equation}
and $\rho (\varepsilon) \to \infty$ as $\varepsilon \to 0$. We use aliases $x
= x^1$, $y = x^2$, $z = x^3$ for the coordinates.

First, we deal with integrals on the edges. There are two type edges. By
rotational symmetry, we can assume one edge is
\[ E_1 = \{(0, \tfrac{1 - \varepsilon z}{1 - \varepsilon^2} \rho, z) : z \in
   [\varepsilon, \varepsilon^{- 1}]\} . \]
The length element $\mathrm{d} \bar{\lambda}$ on $E_1$ is
\[ z^{- 1} \sqrt{1 + \frac{\varepsilon^2 \rho^2}{(1 - \varepsilon^2)^2}}
   \mathrm{d} z \leqslant \sqrt{2} z^{- 1} \max \{C \varepsilon \rho, 1\} \]
since we are take $\varepsilon \to 0$. So
\begin{align}
  & \int_{E_1}^{} \cosh^{- 2 \tau + 1} r \mathrm{d} \bar{\lambda} \\
  \leqslant & C \max \{\varepsilon \rho, 1\}
  \int_{\varepsilon}^{\varepsilon^{- 1}} z^{- 2 + 2 \tau} (z^2 + 1 + (\tfrac{1
  - \varepsilon z}{1 - \varepsilon^2} \rho)^2)^{- 2 \tau + 1} \mathrm{d} z.
\end{align}
We divide the integral into three parts. On $[\varepsilon, 1]$, we have
\begin{align}
  & \max \{C \varepsilon \rho, 1\} \int_{\varepsilon}^1 z^{- 2 + 2 \tau} (z^2
  + 1 + (\tfrac{1 - \varepsilon z}{1 - \varepsilon^2} \rho)^2)^{- 2 \tau + 1}
  \mathrm{d} z \\
  \leqslant & \max \{C \varepsilon \rho, 1\} (\tfrac{1}{1 + \varepsilon}
  \rho)^{- 2 \tau + 2} \int_{\varepsilon}^1 z^{- 2 + 2 \tau} \mathrm{d} z
  \\
  = & o (1) \label{below horosphere} 
\end{align}
since $\tau > \tfrac{3}{2}$. On $[1, \tfrac{\varepsilon^{- 1}}{2}]$, for $1 <
a < 2 \tau - 1$,
\begin{align}
  & \max \{C \varepsilon \rho, 1\} \int_1^{\varepsilon^{- 1} / 2} z^{- 2 + 2
  \tau} (z^2 + 1 + (\tfrac{1 - \varepsilon z}{1 - \varepsilon^2} \rho)^2)^{- 2
  \tau + 1} \mathrm{d} z \\
  = & \max \{C \varepsilon \rho, 1\} \int_1^{\varepsilon^{- 1} / 2} z^{- 2 + 2
  \tau - (2 \tau - 2 + a)} (z^2 + 1 + (\tfrac{1 - \varepsilon z}{1 -
  \varepsilon^2} \rho)^2)^{- 2 \tau + 1 + \tfrac{1}{2} (2 \tau - 2 + a)}
  \mathrm{d} z \\
  \leqslant & C \max \{C \varepsilon \rho, 1\} \rho^{- 4 \tau + 2 + 2 \tau - 2
  + a} \int_1^{\varepsilon^{- 1} / 2} z^{- a} \mathrm{d} z \\
  \leqslant & C \max \{C \varepsilon \rho, 1\} \rho^{- 2 \tau + a}
  \label{lower half} 
\end{align}
which is $o (1)$ as $\rho (\varepsilon) \to \infty$. On
$[\tfrac{\varepsilon^{- 1}}{2}, \varepsilon^{- 1}]$, we have that
\begin{align}
  & \max \{C \varepsilon \rho, 1\} \int_{\varepsilon^{- 1} /
  2}^{\varepsilon^{- 1}} z^{- 2 + 2 \tau} (z^2 + 1 + (\tfrac{1 - \varepsilon
  z}{1 - \varepsilon^2} \rho)^2)^{- 2 \tau + 1} \mathrm{d} z \\
  \leqslant & \max \{C \varepsilon \rho, 1\} \int_{\varepsilon^{- 1} /
  2}^{\varepsilon^{- 1}} z^{- 2 + 2 \tau - 4 \tau + 2} \mathrm{d} z
  \\
  \leqslant & \max \{C \varepsilon \rho, 1\} \varepsilon^{2 \tau - 1} .
  \label{top half} 
\end{align}
Requiring that $\rho (\varepsilon) = o (\varepsilon^{- 2 \tau})$ this term is
$o (1)$. By rotational symmetry again, the other type edges we can just
consider
\begin{equation}
  E_2 = \{(\rho \cos \tfrac{\pi}{n}, y, \varepsilon) : y \in [- \rho \sin
  \tfrac{\pi}{n}, \rho \sin \tfrac{\pi}{n}]\} .
\end{equation}
The integral of $\cosh^{- 2 \tau + 1} r$ on $E_2$ is
\begin{align}
  & \int_{E_2} \cosh^{- 2 \tau + 1} r \mathrm{d} \bar{\lambda} \\
  \leqslant & C \int_{- \rho \sin \tfrac{\pi}{n}}^{\rho \sin \tfrac{\pi}{n}}
  \varepsilon^{- 1} \cdot \varepsilon^{- 1 + 2 \tau} (y^2 + \rho^2 \cos^2
  \tfrac{\pi}{n} + \varepsilon^2 + 1)^{- 2 \tau + 1} \mathrm{d} y \\
  \leqslant & C \varepsilon^{- 2 + 2 \tau} \int_{\mathbb{R}} (x^2 + 1)^{- 2
  \tau + 1} \mathrm{d} y = o (1) 
\end{align}
as $\varepsilon \to 0$.

For the face $F$ of $\mathbf{{\Delta}}$ lying on $\{z = \varepsilon\}$, we
have
\begin{align}
  & \int_F \cosh^{- 2 \tau + 1} r \mathrm{d} \bar{v} \\
  \leqslant & C \int_F \varepsilon^{- 2 - 1 + 2 \tau} (x^2 + y^2 +
  \varepsilon^2 + 1)^{- 2 \tau + 1} \mathrm{d} x \mathrm{d} y \\
  \leqslant & C \varepsilon^{2 \tau - 3} \int_0^{\rho (\varepsilon)} (s^2 +
  1)^{- 2 \tau + 1} s \mathrm{d} s = o (1) 
\end{align}
as $\varepsilon \to 0$. Now we consider a side face, we use the symbol $S$. We
use the Euclidean distance $\xi$ of a point in side edges to the $z$-axis and
$z$ to parametrized $S$. First, $\xi = \tfrac{1 - \varepsilon z}{1 -
\varepsilon^2} \rho$, then
\begin{equation}
  S = \{(\xi \cos \tfrac{\pi}{n}, y, z) : y \in [- \xi \sin \tfrac{\pi}{n},
  \xi \sin \tfrac{\pi}{n}], z \in [\varepsilon, \varepsilon^{- 1}]\} .
\end{equation}
It is easy to show that
\begin{align}
  & \int_S \cosh^{- 2 \tau + 1} r \mathrm{d} \bar{\sigma} \\
  = & \sqrt{1 + \tfrac{\varepsilon^2 \rho^2}{(1 - \varepsilon^2)^2}} 2^{2 \tau
  - 1} \int_{\varepsilon}^{\varepsilon^{- 1}} \int_{- \xi \sin
  \tfrac{\pi}{n}}^{\xi \sin \tfrac{\pi}{n}} z^{2 \tau - 3} (\xi^2 \cos^2
  \tfrac{\pi}{n} + y^2 + 1 + z^2)^{- 2 \tau + 1} \mathrm{d} y \mathrm{d} z
  \\
  \leqslant & \max \{C \varepsilon \rho, C_1 \}
  \int_{\varepsilon}^{\varepsilon^{- 1}} \xi \sin \tfrac{\pi}{n} z^{2 \tau -
  3} (\xi^2 \cos^2 \tfrac{\pi}{n} + 1 + z^2)^{- 2 \tau + 1} \mathrm{d} z
\end{align}
Similar to {\eqref{below horosphere}} and {\eqref{lower half}}, we have that
\begin{equation}
  \max \{C \varepsilon \rho, C_1 \} (\int_{\varepsilon}^1 +
  \int_1^{\varepsilon^{- 1} / 2}) = o (1) .
\end{equation}
For the integral over $[\varepsilon^{- 1} / 2, \varepsilon^{- 1}]$, we absorb
$\xi \sin \tfrac{\pi}{n}$ by $\xi^2 \cos^2 \tfrac{\pi}{n} + z^2 + 1$ and
\begin{align}
  & \max \{C \varepsilon \rho, C_1 \} \int_{\varepsilon^{- 1} /
  2}^{\varepsilon^{- 1}} \xi \sin \tfrac{\pi}{n} z^{2 \tau - 3} (\xi^2 \cos^2
  \tfrac{\pi}{n} + 1 + z^2)^{- 2 \tau + 1} \mathrm{d} z \\
  \leqslant & \max \{C \varepsilon \rho, C_1 \} \int_{\varepsilon^{- 1} /
  2}^{\varepsilon^{- 1}} z^{2 \tau - 3} (\xi^2 \cos^2 \tfrac{\pi}{n} + 1 +
  z^2)^{- 2 \tau + 1 + \tfrac{1}{2}} \mathrm{d} z \\
  \leqslant & \max \{C \varepsilon \rho, C_1 \} \int_{\varepsilon^{- 1} /
  2}^{\varepsilon^{- 1}} z^{2 \tau - 3 - 4 \tau + 2 + 1} \mathrm{d} z
  \\
  \leqslant & \max \{C \varepsilon \rho, C_1 \} \varepsilon^{2 \tau - 1} = o
  (1) . 
\end{align}
where we have required that $\rho (\varepsilon) = o (\varepsilon^{- 2 \tau})$.

\begin{proof}[Proof of Theorem \ref{miao}]
  We have on each face $F$ of $\mathbf{{\Delta}}_q$,
  \begin{equation}
    2 V (H - \bar{H}) = -\mathbb{U}^i \bar{\nu}_i
    -\ensuremath{\operatorname{div}}_F (V X) + O (\mathrm{e}^{- 2 \tau r + r})
    \label{mean curvature expansion}
  \end{equation}
  where $X$ is the vector field dual to the 1-form $e (\bar{\nu}, \cdot)$ with
  respect to the metric $b|_F$. This is an easy consequence of
  {\cite[(2.5)]{jang-hyperbolic-2021}} that
  \begin{align}
    & \mathbb{U} (\bar{\nu}) \\
    = & 2 V (\bar{H} - H) -\ensuremath{\operatorname{div}}_F (V X) +
    [(\ensuremath{\operatorname{tr}}_b e - e (\bar{\nu}, \bar{\nu})) \langle
    \mathrm{d} V, \bar{\nu} \rangle - V \langle \bar{A}, e \rangle_b] + O
    (\mathrm{e}^{- 2 \tau r + r}) . 
\end{align}
  From {\eqref{normal derivative of static potential}} and {\eqref{second
  fundamental form background}}, we obtain the desired formula {\eqref{mean
  curvature expansion}}. From {\eqref{mean curvature expansion}}, we have that
  \begin{equation}
    \int_{\partial \mathbf{{\Delta}}_q} \mathbb{U}^i \bar{\nu}_i \mathrm{d}
    \bar{\sigma} = \int_{\partial \mathbf{{\Delta}}_q} [- 2 V (H - \bar{H})
    -\ensuremath{\operatorname{div}}_{\partial \mathbf{{\Delta}}_q} (V X)]
    \mathrm{d} \bar{\sigma} + o (1)
  \end{equation}
  On each face $F$, using divergence theorem
  \begin{equation}
    \int_F \ensuremath{\operatorname{div}}_F (V X) \mathrm{d} \bar{\sigma} =
    \int_{\partial F} V e (\bar{\nu}, \bar{n}) \mathrm{d} \bar{\lambda},
  \end{equation}
  where $\bar{n}$ is the $b$-normal to $\partial F$ in $F$. On the edge $F_A
  \cap F_B$, the contribution is
  \begin{equation}
    \int_{F_A \cap F_B} V [e (\bar{\nu}_A, \bar{n}_A) + e (\bar{\nu}_B,
    \bar{n}_B)] \mathrm{d} \bar{\lambda} .
  \end{equation}
  Let $g_{i j} = g (\partial_i, \partial_j)$, we have that
  \begin{equation}
    \varepsilon_{i j} := (x^n)^2 e_{i j} = (x^n)^2 g_{i j} - (x^n)^2
    \bar{g}_{i j} = (x^n)^2 \bar{g}_{i j} - \delta_{i j} = O (\mathrm{e}^{-
    \tau r})
  \end{equation}
  The $g$-normal to the face $F_A$ is then expressed as
  \begin{equation}
    \nu_A = \frac{g^{i j} a_j \partial_j}{\sqrt{g^{k l} a_k a_l}} .
  \end{equation}
  We write
  \begin{align}
    \cos \theta = & g (\nu_A, \nu_B) \\
    = & (a_i b_j g^{i j}) (g^{k l} a_k a_l)^{- \tfrac{1}{2}} (g^{p q} a_p
    a_q)^{- \tfrac{1}{2}} \\
    = & a_i b_j \tfrac{g^{i j}}{(x^n)^2} (\tfrac{g^{k l}}{(x^n)^2} a_k a_l)^{-
    \tfrac{1}{2}} (\tfrac{g^{p q}}{(x^n)^2} a_p a_q)^{- \tfrac{1}{2}} .
\end{align}
  Up to here, it follows from the same lines as in
  {\cite[(3.9)-(3.30)]{miao-mass-2021}} to show that
  \begin{align}
    & \int_{F_A \cap F_B} V [e (\bar{\nu}_A, \bar{n}_A) + e (\bar{\nu}_B,
    \bar{n}_B)] \mathrm{d} \bar{\lambda} \\
    = & \int_{F_A \cap F_B} V (\bar{\alpha} - \alpha) + \int_{F_A \cap F_B} V
    \cosh^{- 2 \tau} r \mathrm{d} \bar{\lambda} . 
\end{align}
  Therefore,
  \begin{align}
    & \mathbf{M} (V) \\
    = & - \int_{\partial \mathbf{{\Delta}}_q} 2 V (H - \bar{H}) \mathrm{d}
    \bar{\sigma} + 2 \int_{E_q} V (\alpha - \bar{\alpha}) \mathrm{d}
    \bar{\lambda} \\
    & + \int_{\partial \mathbf{{\Delta}}_q} \cosh^{- 2 \tau + 1} r \mathrm{d}
    \bar{\sigma} + \int_{E_q} \cosh^{- 2 \tau + 1} r \mathrm{d} \bar{\lambda}
    + o (1), 
\end{align}
  obtaining the theorem.
\end{proof}

\end{document}